\documentclass[12pt]{article}
\usepackage{amsmath,amscd,amsthm,a4,amssymb}

\def\dual{\,^{^{\complement}}\!}
\def\Rn{{\mathbb{R}^n}}

\def\i{\infty}
\def\a{\alpha}

 \newtheorem{thm}{Theorem}[section]
 \newtheorem{cor}[thm]{Corollary}
 \newtheorem{lem}[thm]{Lemma}
 
 \theoremstyle{definition}
 \newtheorem{defn}[thm]{Definition}
 \theoremstyle{remark}
 \newtheorem{rem}[thm]{Remark}
 
 \numberwithin{equation}{section}

\usepackage{amssymb}
\usepackage{color}

\def\Rn{{\mathbb{R}^n}}
\def\a {\alpha}
\def\i{\infty}

\def\L1loc{L_{\Phi}^{\rm loc}(\Rn)}

\def\dual{\,^{^{\complement}}\!}

\newcommand{\es}{\mathop{\rm ess \; inf}\limits}
\begin{document}

\begin{center}
\LARGE Boundedness of fractional maximal operator and its commutators on generalized Orlicz-Morrey spaces
\end{center}

\

\centerline{\large Vagif S. Guliyev$^{a,b,}$\footnote{
{The research of V. Guliyev and F. Deringoz were partially supported by the grant of Ahi Evran University Scientific Research Projects (PYO.FEN.4003.13.003) and (PYO.FEN.4003-2.13.007).}
\\
E-mail adresses: vagif@guliyev.com (V.S. Guliyev), fderingoz@ahievran.edu.tr (F. Deringoz).}, Fatih Deringoz$^{a,1}$}

\

\centerline{$^{a}$\it Department of Mathematics, Ahi Evran University, Kirsehir, Turkey}

\centerline{$^{b}$\it Institute of Mathematics and Mechanics, Baku, Azerbaijan}

\

\begin{abstract}
We consider generalized Orlicz-Morrey spaces $M_{\Phi,\varphi}(\Rn)$ including their weak versions $WM_{\Phi,\varphi}(\Rn)$. We find the sufficient conditions on the pairs $(\varphi_{1},\varphi_{2})$ and $(\Phi, \Psi)$ which ensures the boundedness of the fractional maximal operator $M_{\a}$ from $M_{\Phi,\varphi_1}(\Rn)$ to $M_{\Psi,\varphi_2}(\Rn)$ and from $M_{\Phi,\varphi_1}(\Rn)$ to \linebreak $WM_{\Psi,\varphi_2}(\Rn)$. As applications of those results, the boundedness of the commutators of the fractional maximal operator $M_{b,\a}$
with $b \in BMO(\Rn)$ on the spaces $M_{\Phi,\varphi}(\Rn)$ is also obtained. In all the cases the conditions for the boundedness are given in terms of supremal-type inequalities on weights $\varphi(x,r)$, which do not assume any assumption on monotonicity of $\varphi(x,r)$ on $r$.
\end{abstract}

\

\



\

\section{Introduction}
Boundedness of classical operators of the Real analysis, such as the maximal operator, fractional maximal operator, Riesz potential and the singular integral operators etc, have been
extensively investigated in various function spaces. Results on weak and strong type inequalities for operators of this kind in Lebesgue spaces are classical and can be found for example in \cite{BenSharp, St, Torch}. These boundedness extended to several function spaces which are generalizations of $L_{p}$-spaces, for example, Orlicz spaces, Morrey spaces, Lorentz spaces, Herz spaces, etc.

Orlicz spaces, introduced in \cite{Orlicz1, Orlicz2}, are generalizations of Lebesgue spaces $L_p$. They are useful tools in harmonic analysis and its applications. For example, the Hardy-Littlewood maximal operator is bounded on $L_p$ for $1 < p < \infty$, but not on $L_1$. Using Orlicz spaces, we can investigate the boundedness of
the maximal operator near $p = 1$ more precisely (see  \cite{Cianchi1, Kita1, Kita2}).

On the other hand, Morrey spaces were introduced in \cite{Morrey} to estimate solutions of partial differential equations, and studied by many authors.

Let $f\in L_1^{\rm loc}(\Rn)$. The fractional maximal operator $M_{\a}$ and the Riesz potential operator $I_{\a}$ are  defined by
$$
M_{\a}f(x)=\sup_{t>0}|B(x,t)|^{-1+ \frac{\a}{n}}\int_{B(x,t)} |f(y)|dy,  ~~ 0 \le \a < n,
$$
$$
I_{\a}f(x)=\int_{\Rn} \frac{f(y)}{|x-y|^{n-\a}}dy,  ~~ 0 < \a < n.
$$
If $\a=0$, then $M \equiv M_{0}$ is the Hardy-Littlewood maximal operator.

The operator $M_{\a}$ is of weak type $(p,np/(n-\a p))$ if $1\le p\le n/\a$ and of strong type $(p,np/(n-\a p))$ if $1< p\le n/\a$.
Also the operator $I_{\a}$ is of weak type $(p,np/(n-\a p))$ if $1\le p< n/\a$ and of strong type $(p,np/(n-\a p))$ if $1< p< n/\a$.

The boundedness of $M_{\a}$ and $I_{\a}$ from Orlicz space $L_{\Phi}(\Rn)$ to $L_{\Psi}(\Rn)$ was studied by Cianchi \cite{Cianchi1}. For boundedness of $M_{\a}$ and $I_{\a}$ on Morrey spaces $M_{p,\lambda}(\mathbb{R}^{n})$, see Peetre (Spanne)\cite{Peetre}, Adams \cite{Adams}.

The definition of generalized Orlicz-Morrey spaces introduced in  \cite{DGS} and used
here is different from that of Sawano et al. \cite{SawSugTan} and Nakai \cite{Nakai3, Nakai1}.

In \cite{DGS}, the boundedness of the maximal operator $M$
and the Calder\'{o}n-Zygmund operator $T$ from one generalized Orlicz-Morrey space $M_{\Phi,\varphi_1}(\Rn)$ to $M_{\Phi,\varphi_2}(\Rn)$ and from $M_{\Phi,\varphi_1}(\Rn)$ to the weak space $WM_{\Phi,\varphi_2}(\Rn)$ was proved (see, also \cite{GulDerHasJIA}).
Also in \cite{GulDerJFSA} the authors prove the boundedness of the Riesz potential operator $I_{\a}$ and its commutator $[b,I_{\a}]$ from  $M_{\Phi,\varphi_1}(\Rn)$ to $M_{\Psi,\varphi_2}(\Rn)$ and from $M_{\Phi,\varphi_1}(\Rn)$ to $WM_{\Psi,\varphi_2}(\Rn)$.

The main purpose of this paper is to find sufficient conditions on the general Young functions $\Phi, \Psi$ and the functions $\varphi_1$, $\varphi_2$ which ensure the boundedness of $M_{\a}$ from $M_{\Phi,\varphi_1}(\Rn)$ to  $M_{\Psi,\varphi_2}(\Rn)$, from $M_{\Phi,\varphi_1}(\Rn)$ to $WM_{\Psi,\varphi_2}(\Rn)$ and in the case $b \in BMO$ the boundedness of the commutator of the fractional maximal operator $M_{b,\a}$ from $M_{\Phi,\varphi_1}(\Rn)$ to $M_{\Psi,\varphi_2}(\Rn)$.

In the next section we recall the definitions of Orlicz and Morrey spaces and give the definition of generalized Orlicz-Morrey spaces in Section 3. In Section 4 and Section 5 the results on boundedness of $M_{\a}$ and its commutator operator $M_{b,\a}$ is obtained.

By $A\lesssim B$ we mean that $A\le CB$ with some positive constant $C$ independent
of appropriate quantities. If $A\lesssim B$ and $B\lesssim A$, we write $A\thickapprox B$ and say that $A$
and $B$ are equivalent.

\

\section{Some preliminaries on Orlicz and Morrey spaces}

In the study of local properties of solutions to of partial differential equations, together with weighted Lebesgue spaces,  Morrey spaces $M_{p,\lambda}(\Rn)$ play an important role, see \cite{Gi}. Introduced by C. Morrey  \cite{Morrey} in 1938, they are defined by the norm
\begin{equation*}
\left\| f\right\|_{M_{p,\lambda}}: = \sup_{x, \; r>0 } r^{-\frac{\lambda}{p}} \|f\|_{L_{p}(B(x,r))},
\end{equation*}
where $0 \le \lambda \le  n,$ $1\le p < \infty.$ Here and everywhere in the sequel $B(x,r)$ stands for the
ball in $\mathbb{R}^n$ of radius $r$ centered at $x$. Let $|B(x,r)|$ be the Lebesgue measure of the ball $B(x,r)$ and $|B(x,r)|=v_n r^n$, where $v_n=|B(0,1)|$.

Note that $M_{p,0}=L_{p}(\Rn)$ and $M_{p,n}=L_{\infty}(\Rn)$. If $\lambda<0$ or $\lambda>n$, then $M_{p,\lambda }={\Theta}$,
where $\Theta$ is the set of all functions equivalent to $0$ on $\Rn$.

We also denote  by $WM_{p,\lambda} \equiv WM_{p,\lambda}(\Rn)$ the weak Morrey space of all functions $f\in WL_{p}^{\rm loc}(\Rn) $ for which
$$
\left\| f\right\|_{WM_{p,\lambda }} = \sup_{x\in \Rn, \; r>0} r^{-\frac{\lambda}{p}} \|f\|_{WL_{p}(B(x,r))} <\infty,
$$
where $WL_{p}(B(x,r))$ denotes the weak $L_p$-space.

We refer in particular to \cite{KJF} for the classical Morrey spaces.

We recall the definition of Young functions.

\begin{defn}\label{def2} A function $\Phi : [0,+\infty) \rightarrow [0,\infty]$ is called a Young function if $\Phi$ is convex, left-continuous, $\lim\limits_{r\rightarrow +0} \Phi(r) = \Phi(0) = 0$ and $\lim\limits_{r\rightarrow +\infty} \Phi(r) = \infty$.
\end{defn}

From the convexity and $\Phi(0) = 0$ it follows that any Young function is increasing.
If there exists $s \in (0,+\infty)$ such that $\Phi(s) = +\infty$, then $\Phi(r) = +\infty$ for $r \geq s$.

%
%
Let $\mathcal{Y}$ be the set of all Young functions $\Phi$ such that
\begin{equation*}
0<\Phi(r)<+\infty\qquad \text{for} \qquad 0<r<+\infty
\end{equation*}
If $\Phi \in \mathcal{Y}$, then $\Phi$ is absolutely continuous on every closed interval in $[0,+\infty)$
and bijective from $[0,+\infty)$ to itself.

\begin{defn} (Orlicz Space). For a Young function $\Phi$, the set
$$L_{\Phi}(\Rn)=\left\{f\in L_1^{\rm loc}(\Rn): \int_{\Rn}\Phi(k|f(x)|)dx<+\infty
 \text{ for some $k>0$  }\right\}$$
is called Orlicz space. If $\Phi(r)=r^{p},\, 1\le p<\i$, then $L_{\Phi}(\Rn)=L_{p}(\Rn)$. If $\Phi(r)=0,\,(0\le r\le 1)$ and $\Phi(r)=\i,\,(r> 1)$, then $L_{\Phi}(\Rn)=L_{\i}(\Rn)$.
The  space $L_{\Phi}^{\rm loc}(\Rn)$ endowed with the natural topology  is defined as the set of all functions $f$ such that  $f\chi_{_B}\in L_{\Phi}(\Rn)$ for all balls $B \subset \Rn$.
We refer to the books \cite{KokKrbec, KrasnRut, RaoRen} for the theory of Orlicz Spaces.
\end{defn}

Note that, $L_{\Phi}(\Rn)$ is a Banach space with respect to the norm
$$\|f\|_{L_{\Phi}}=\inf\left\{\lambda>0:\int_{\Rn}\Phi\Big(\frac{|f(x)|}{\lambda}\Big)dx\leq 1\right\},$$
 so  that
$$\int_{\Rn}\Phi\Big(\frac{|f(x)|}{\|f\|_{L_{\Phi}}}\Big)dx\leq 1.$$
For a measurable set $\Omega\subset \mathbb{R}^{n}$, a measurable function $f$ and $t>0$, let
$$
m(\Omega,\ f,\ t)=|\{x\in\Omega:|f(x)|>t\}|.
$$
In the case $\Omega=\mathbb{R}^{n}$, we shortly denote it by $m(f,\ t)$.
\begin{defn} The weak Orlicz space
$$
WL_{\Phi}(\mathbb{R}^{n}):=\{f\in L^{\rm loc}_{1}(\mathbb{R}^{n}):\Vert f\Vert_{WL_{\Phi}}<+\infty\}
$$
is defined by the norm
$$
\Vert f\Vert_{WL_{\Phi}}=\inf\Big\{\lambda>0\ :\ \sup_{t>0}\Phi(t)m\Big(\frac{f}{\lambda},\ t\Big)\ \leq 1\Big\}.
$$
\end{defn}


For a Young function $\Phi$ and  $0 \leq s \leq +\infty$, let
$$\Phi^{-1}(s)=\inf\{r\geq 0: \Phi(r)>s\}\qquad (\inf\emptyset=+\infty).$$
If $\Phi \in \mathcal{Y}$, then $\Phi^{-1}$ is the usual inverse function of $\Phi$. We note that
\begin{equation}\label{younginverse}
\Phi(\Phi^{-1}(r))\leq r \leq \Phi^{-1}(\Phi(r)) \quad \text{ for } 0\leq r<+\infty.
\end{equation}
A Young function $\Phi$ is said to satisfy the $\Delta_2$-condition, denoted by  $\Phi \in \Delta_2$, if
$$
\Phi(2r)\le k\Phi(r) \text{    for } r>0
$$
for some $k>1$. If $\Phi \in \Delta_2$, then $\Phi \in \mathcal{Y}$. A Young function $\Phi$ is said to satisfy the $\nabla_2$-condition, denoted also by  $\Phi \in \nabla_2$, if
$$\Phi(r)\leq \frac{1}{2k}\Phi(kr),\qquad r\geq 0,$$
for some $k>1$. The function $\Phi(r) = r$ satisfies the $\Delta_2$-condition but does not satisfy the $\nabla_2$-condition.
If $1 < p < \infty$, then $\Phi(r) = r^p$ satisfies both the conditions. The function $\Phi(r) = e^r - r - 1$ satisfies the
$\nabla_2$-condition but does not satisfy the $\Delta_2$-condition.

For a Young function $\Phi$, the complementary function $\widetilde{\Phi}(r)$ is defined by
\begin{equation*}
\widetilde{\Phi}(r)=\left\{
\begin{array}{ccc}
\sup\{rs-\Phi(s): s\in [0,\infty)\}
& , & r\in [0,\infty) \\
+\infty&,& r=+\infty.
\end{array}
\right.
\end{equation*}
The complementary function  $\widetilde{\Phi}$ is also a Young function and $\widetilde{\widetilde{\Phi}}=\Phi$. If $\Phi(r)=r$, then $\widetilde{\Phi}(r)=0$ for $0\leq r \leq 1$ and $\widetilde{\Phi}(r)=+\infty$
  for $r>1$. If $1 < p < \infty$, $1/p+1/p^\prime= 1$ and $\Phi(r) =
r^p/p$, then $\widetilde{\Phi}(r) = r^{p^\prime}/p^\prime$. If $\Phi(r) = e^r-r-1$, then $\widetilde{\Phi}(r) = (1+r) \log(1+r)-r$. Note that $\Phi \in \nabla_2$ if and only if $\widetilde{\Phi} \in \Delta_2$. It is known that
\begin{equation}\label{2.3}
r\leq \Phi^{-1}(r)\widetilde{\Phi}^{-1}(r)\leq 2r \qquad \text{for } r\geq 0.
\end{equation}

Note that Young functions satisfy the  properties
\begin{equation*}
\Phi(\a t)\leq \a \Phi(t)
\end{equation*}
for  all $0\le\a\le1$ and  $0 \le t < \i$, and
\begin{equation*}
\Phi(\beta t)\geq \beta \Phi(t)
\end{equation*}
for  all $\beta>1$ and  $0 \le t < \i$.

The following analogue of the H\"older inequality is known, see  \cite{Weiss}.
\begin{thm} \cite{Weiss} \label{HolderOr}
For a Young function $\Phi$ and its complementary function  $\widetilde{\Phi}$,
the following inequality is valid
$$\|fg\|_{L_{1}(\Rn)} \leq 2 \|f\|_{L_{\Phi}} \|g\|_{L_{\widetilde{\Phi}}}.$$
\end{thm}

The following lemma is valid.
\begin{lem}\label{lem4.0}  \cite{BenSharp, LiuWang}
Let $\Phi$ be a Young function and $B$ a set in $\mathbb{R}^n$ with finite Lebesgue measure. Then
$$
\|\chi_{_B}\|_{WL_{\Phi}(\Rn)} = \|\chi_{_B}\|_{L_{\Phi}(\Rn)} = \frac{1}{\Phi^{-1}\left(|B|^{-1}\right)}.
$$
\end{lem}

In the next sections where we prove our main estimates, we use the following lemma, which follows  from Theorem \ref{HolderOr}, Lemma \ref{lem4.0} and the inequality \eqref{2.3}.
\begin{lem}\label{lemHold}
For a Young function $\Phi$ and $B=B(x,r)$, the following inequality is valid
$$\|f\|_{L_{1}(B)} \leq 2 |B| \Phi^{-1}\left(|B|^{-1}\right) \|f\|_{L_{\Phi}(B)} .$$
\end{lem}

\

Necessary and sufficient conditions on $(\Phi, \Psi)$ for the boundedness of $M_{\a}$ and $I_{\a}$ from Orlicz spaces $L_{\Phi}(\Rn)$ to $L_{\Psi}(\Rn)$ and $L_{\Phi}(\Rn)$ to $WL_{\Psi}(\Rn)$
have been obtained in \cite[ Theorem 1 and 2]{Cianchi1}. In the statement of the theorems, $\Psi_{p}$ is the Young function associated with the Young function $\Psi$ and $p\in(1,\i]$ whose Young conjugate is given by
\begin{equation}\label{chi2.3}
\widetilde{\Psi_{p}}(s)=\int_{0}^{s}r^{p^{\prime}-1}(\mathcal{B}_{p}^{-1}(r^{p^{\prime}}))^{p^{\prime}}dr,
\end{equation}
where
\begin{center}
$\mathcal{B}_{p}(s)= \int_{0}^{s}\frac{\Psi(t)}{t^{1+p^{\prime}}}dt$
\end{center}
and $p^{\prime}$, the Holder conjugate of $p$, equals either $p/(p-1)$ or 1, according to whether $ p<\infty$ or $ p=\infty$ and $\Phi_{p}$ denotes the Young function defined by
\begin{equation}\label{chi2.13}
\Phi_{p}(s)= \int_{0}^{s}r^{p^{\prime}-1}(\mathcal{A}_{p}^{-1}(r^{p'}))^{p'}dr,
\end{equation}
where
\begin{center}
$\mathcal{A}_{p}(s)= \int_{0}^{s}\frac{\widetilde{\Phi}(t)}{t^{1+p^{\prime}}}dt$.
\end{center}

Recall that, if $\Phi$ and $\Psi$ are functions from $[0,\i)$ into $[0,\i]$, then $\Psi$ is said to dominate $\Phi$ globally if a positive constant $c$ exists such that $\Phi(s)\le\Psi(cs)$ for all $s\geq0$.

\begin{thm}\label{bounFrMaxOrl} \cite{Cianchi1}

(i)The fractional maximal operator $M_{\a}$ is bounded from $L_{\Phi}(\Rn)$ to $WL_{\Psi}(\Rn)$ if and only if
\begin{equation}\label{condweakFrM}
\text{$\Phi$ dominates globally the function $Q$,}
\end{equation}
whose inverse is given by
$$
Q^{-1}(r)=r^{\a/n}\Psi^{-1}(r).
$$

(ii) The fractional maximal operator $M_{\a}$ is bounded from $L_{\Phi}(\Rn)$ to $L_{\Psi}(\Rn)$ if and only if
\begin{equation}\label{condstrFrM}
\int_{0}^{1} \frac{\Psi(t)}{t^{1+n/(n-\a)}}dt<\i \text{  and $\Phi$ dominates globally the function $\Psi_{n/\a}$}.
\end{equation}
\end{thm}

\begin{thm}\label{bounPotOrl} \cite{Cianchi1}
Let $0<\a<n$. Let $\Phi$ and $\Psi$ Young functions and let $\Phi_{n/\alpha}$ and $\Psi_{n/\alpha}$ be the Young functions defined as in \eqref{chi2.13} and \eqref{chi2.3}, respectively. Then

(i)The Riesz potential $I_{\a}$ is bounded from $L_{\Phi}(\Rn)$ to $WL_{\Psi}(\Rn)$ if and only if
\begin{equation}\label{condweakpot}
\int_{0}^{1}\widetilde{\Phi}(t)/t^{1+n/(n-\alpha)}dt<\infty \text{  and $\Phi_{n/\alpha}$ dominates $\Psi$ globally.}
\end{equation}

(ii) The Riesz potential $I_{\a}$ is bounded from $L_{\Phi}(\Rn)$ to $L_{\Psi}(\Rn)$ if and only if
\begin{align}\label{condstrpot}
&\int_{0}^{1}\widetilde{\Phi}(t)/t^{1+n/(n-\alpha)}dt<\i, ~~ \int_{0}^{1}\Psi(t)/t^{1+n/(n-\alpha)}dt<\i, \notag
\\
& \text{ $\Phi$ dominates $\Psi_{n/\alpha}$  globally and $\Phi_{n/\alpha}$ dominates $\Psi$ globally.}
\end{align}
\end{thm}

\

\section{Generalized Orlicz-Morrey Spaces}

\begin{defn} (Orlicz-Morrey Space). For a Young function $\Phi$ and $0 \le \lambda \le n$,
we denote by $M_{\Phi,\lambda}(\Rn)$ the Orlicz-Morrey space, the space of all
functions $f\in L_{\Phi}^{\rm loc}(\Rn)$ with finite quasinorm
$$
  \left\| f\right\|_{M_{\Phi,\lambda}}= \sup_{x\in \Rn, \; r>0 }
   \Phi^{-1}\big(r^{-\lambda}\big) \|f\|_{L_{\Phi}(B(x,r))}.
$$
\end{defn}
Note that $M_{\Phi,0}=L_{\Phi}(\Rn)$ and if $\Phi(r)=r^{p},\,1\le p<\i$, then $M_{\Phi,\lambda}(\Rn)=M_{p,\lambda}(\Rn)$.

We also denote  by $WM_{\Phi,\lambda}(\Rn)$ the weak Orlicz-Morrey space of all functions $f\in WL_{\Phi}^{\rm loc}(\Rn) $ for which
$$
\left\| f\right\|_{WM_{\Phi,\lambda }} = \sup_{x\in \Rn, \; r>0 }
\Phi^{-1}\big(r^{-\lambda}\big) \|f\|_{WL_{\Phi}(B(x,r))} <\infty,
$$
where $WL_{\Phi}(B(x,r))$ denotes the weak $L_\Phi$-space of measurable functions
$f$ for which
\begin{align*}
\|f\|_{WL_{\Phi}(B(x,r))} & \equiv \|f \chi_{_{B(x,r)}}\|_{WL_{\Phi}(\Rn)}.
\end{align*}

\begin{defn} (generalized Orlicz-Morrey Space)
Let $\varphi(x,r)$ be a positive measurable function on $\Rn \times (0,\infty)$ and $\Phi$ any Young function.
We denote by $M_{\Phi,\varphi}(\Rn)$ the generalized Orlicz-Morrey space, the space of all
functions $f\in L_{\Phi}^{\rm loc}(\Rn)$ with finite quasinorm
$$
\|f\|_{M_{\Phi,\varphi}} = \sup\limits_{x\in\Rn, r>0}
\varphi(x,r)^{-1} \Phi^{-1}(r^{-n}) \|f\|_{L_{\Phi}(B(x,r))}.
$$
Also by $WM_{\Phi,\varphi}(\Rn)$ we denote the weak generalized Orlicz-Morrey space of all functions $f\in WL_{\Phi}^{\rm loc}(\Rn)$ for which
$$
\|f\|_{WM_{p,\varphi}} = \sup\limits_{x\in\Rn, r>0} \varphi(x,r)^{-1} \Phi^{-1}(r^{-n}) \|f\|_{WL_{\Phi}(B(x,r))} < \infty.
$$
\end{defn}

According to this definition, we recover the spaces $M_{\Phi,\lambda}$ and $WM_{\Phi,\lambda}$ under the choice $\varphi(x,r)=\frac{\Phi^{-1}\big(r^{-n}\big)}{\Phi^{-1}\big(r^{-\lambda}\big)}$:
\begin{align*}
M_{\Phi,\lambda} = M_{\Phi,\varphi}\Big|_{\varphi(x,r)=\frac{\Phi^{-1}\big(r^{-n}\big)}{\Phi^{-1}\big(r^{-\lambda}\big)}},
~~~~  WM_{\Phi,\lambda} = WM_{\Phi,\varphi}\Big|_{\frac{\Phi^{-1}\big(r^{-n}\big)}{\Phi^{-1}\big(r^{-\lambda}\big)}}.
\end{align*}

According to this definition, we recover the generalized Morrey spaces $M_{p,\varphi}$ and weak generalized Morrey spaces $WM_{p,\varphi}$ under the choice $\Phi(r)=r^{p},\,1\le p<\i$:
\begin{align*}
M_{p,\varphi}=
M_{\Phi,\varphi}\Big|_{\Phi(r)=r^{p}},
~~~~
WM_{p,\varphi}=WM_{\Phi,\varphi}\Big|_{\Phi(r)=r^{p}}.
\end{align*}

Sufficient conditions on $\varphi$ for the boundedness of $M_{\a}$ and $I_{\a}$ in generalized Morrey spaces $M_{p,\varphi}(\Rn)$ have been obtained in
\cite{BurGogGulMust1, GulDoc, GulBook, GulJIA, GulHasSam, GulShu, Miz, Nakai}.

\

\section{ Boundedness of the fractional maximal operator in the spaces $M_{\Phi,\varphi}(\Rn)$}

In this section sufficient conditions on the pairs $(\varphi_1, \varphi_2)$ and $(\Phi, \Psi)$ for the boundedness of $M_{\a}$ from one generalized Orlicz-Morrey spaces $M_{\Phi,\varphi_1}(\Rn)$ to another $M_{\Psi,\varphi_2}(\Rn)$ and from $M_{\Phi,\varphi_1}(\Rn)$ to the weak space $WM_{\Psi,\varphi_2}(\Rn)$  have been obtained.
At first we recall some supremal inequalities which we use at the proof of our main theorem.

Let $v$ be a weight. We denote by $L_{\infty,v}(0,\infty)$ the space of all functions $g(t)$, $t>0$ with finite norm
$$
\|g\|_{L_{\infty,v}(0,\infty)} = \sup _{t>0}v(t)|g(t)|
$$
and  $L_{\infty}(0,\infty) \equiv L_{\infty,1}(0,\infty)$.
Let ${\mathfrak M}(0,\infty)$ be the set of all \linebreak  Lebesgue-measurable
functions on $(0,\infty)$ and ${\mathfrak M}^+(0,\infty)$ its subset
 of all nonnegative functions on $(0,\infty)$. We denote by
${\mathfrak M}^+\!(0,\infty;\!\uparrow\!)\!$ the cone of all functions in
${\mathfrak M}^+(0,\infty)$ which are non-decreasing on $(0,\infty)$ and
$$
\mathcal{A}=\left\{\varphi \in {\mathfrak M}^+(0,\infty;\uparrow):
\lim_{t\rightarrow 0+}\varphi(t)=0\right\}.
$$
Let $u$ be a continuous and non-negative function on $(0,\infty)$. We
define the supremal operator $\overline{S}_{u}$ on $g\in {\mathfrak M}(0,\infty)$ by
$$
(\overline{S}_{u}g)(t): = \|u\, g\|_{L_{\infty}(t,\infty)},~~t\in (0,\infty).
$$
The following theorem was proved in \cite{BurGogGulMust1}.
\begin{thm}\label{thm5.1}
Let $v_1$, $v_2$ be non-negative measurable functions satisfying
$0<\|v_1\|_{L_{\infty}(t,\infty)}<\infty$ for any $t>0$ and let $u$ be a continuous non-negative function on $(0,\infty).$
Then the operator $\overline{S}_{u}$ is bounded from
$L_{\infty,v_1}(0,\infty)$ to $L_{\infty,v_2}(0,\infty)$ on the cone $\mathcal{A}$ if and only if
\begin{equation}\label{eq6.8}
\begin{split}
\left\|v_2 \overline{S}_{u}\left(  \| v_1 \|^{-1}_{L_{\infty}(\cdot,\infty)}\right)\right\|_{L_{\infty}(0,\infty)}<\infty.
\end{split}
\end{equation}
\end{thm}

For the Riesz potential the following local estimate was proved in \cite{GulDerJFSA}.
\begin{lem}\label{lemGultecnOrl}
Let $0<\a<n$, $\Phi$ and $\Psi$ Young functions, $f\in L_{\Phi}^{\rm loc}(\Rn)$ and $B=B(x_0,r)$. If $(\Phi, \Psi)$ satisfy the conditions \eqref{condstrpot}, then
\begin{equation}\label{strGultecnOrl}
\|I_\alpha f\|_{L_{\Psi}(B)} \lesssim \frac{1}{\Psi^{-1}\big(r^{-n}\big)}
 \int_{2r}^{\i} \|f\|_{L_{\Phi}(B(x_0,t))} \Psi^{-1}\big(t^{-n}\big) \frac{dt}{t}.
\end{equation}
If $(\Phi, \Psi)$ satisfy the conditions \eqref{condweakpot}, then
\begin{equation}\label{weakGultecnOrl}
\|I_\alpha f\|_{WL_{\Psi}(B)} \lesssim \frac{1}{\Psi^{-1}\big(r^{-n}\big)}
 \int_{2r}^{\i} \|f\|_{L_{\Phi}(B(x_0,t))} \Psi^{-1}\big(t^{-n}\big) \frac{dt}{t}.
\end{equation}
\end{lem}

For the fractional maximal operator the following local estimate is valid.
\begin{lem}\label{lem4.1}
Let $\Phi$ and $\Psi$ Young functions, $f\in L_{\Phi}^{\rm loc}(\Rn)$ and $B=B(x,r)$.

If $(\Phi, \Psi)$ satisfy the conditions \eqref{condweakFrM}, then
\begin{align}\label{eq100WZ}
\|M_{\a} f\|_{WL_{\Psi}(B)} \lesssim \|f\|_{L_{\Phi}(B(x,2r))} + \frac{1}{\Psi^{-1}\big(r^{-n}\big)} \, \sup_{t>2r} t^{-n+\a} \|f\|_{L_1(B(x,t))}.
\end{align}

If $(\Phi, \Psi)$ satisfy the conditions \eqref{condstrFrM}, then
\begin{align}\label{eq100}
\|M_{\a} f\|_{L_{\Psi}(B)} \lesssim \|f\|_{L_{\Phi}(B(x,2r))} + \frac{1}{\Psi^{-1}\big(r^{-n}\big)} \, \sup_{t>2r} t^{-n+\a} \|f\|_{L_1(B(x,t))}.
\end{align}
\end{lem}
\begin{proof}
Let $(\Phi, \Psi)$ satisfy the conditions \eqref{condstrFrM}. We put $f=f_1+f_2$, where $f_1=f\chi_{B(x,2r)}$ and $f_2=f\chi_{{\dual}B(x,2r)}$. Then we get
\begin{gather*}
\|M_{\a} f\|_{L_{\Psi}(B)} \leq \|M_{\a} f_1\|_{L_{\Psi}(B)}+ \|M_{\a} f_2\|_{L_{\Psi}(B)}.
\end{gather*}
By the boundedness of the operator $M_{\a}$ from  $L_{\Phi}(\Rn)$ to $L_{\Psi}(\Rn)$ (see Theorem \ref{bounFrMaxOrl}) we have
$$
\|M_{\a} f_1\|_{L_{\Psi}(B)}\lesssim \|f\|_{L_{\Phi}(B(x,2r))}.
$$
Let $y$ be an arbitrary point from $B$. If $B(y,t)\cap {\dual}(B(x,2r))\neq\emptyset,$ then $t>r$. Indeed, if $z\in B(y,t)\cap  {\dual} (B(x,2r)),$
then $t > |y-z| \geq |x-z|-|x-y|>2r-r=r$.

On the other hand, $B(y,t)\cap {\dual} (B(x,2r))\subset B(x,2t)$. Indeed, if  $z\in B(y,t)\cap {\dual} (B(x,2r))$, then
we get $|x-z|\leq |y-z|+|x-y|<t+r<2t$.

Hence
\begin{equation*}
\begin{split}
M_{\a} f_2(y) & = \sup_{t>0}\frac{1}{|B(y,t)|^{1-\frac{\a}{n}}} \int_{B(y,t)\cap {{\dual}(B(x,2r))}}|f(z)|d z
\\
& \le 2^{n-\a} \, \sup_{t>r}\frac{1}{|B(x,2t)|^{1-\frac{\a}{n}}} \int_{B(x,2t)}|f(z)|d z
\\
&= 2^{n-\a} \, \sup_{t>2r} \frac{1}{|B(x,t)|^{1-\frac{\a}{n}}} \int_{B(x,t)}|f(z)| d z.
\end{split}
\end{equation*}
Therefore, for all $y \in B$ we have
\begin{equation}\label{ves1}
M_{\a} f_2(y) \le 2^{n-\a} \,\sup_{t>2r} \frac{1}{|B(x,t)|^{1-\frac{\a}{n}}} \int_{B(x,t)} |f(z)|d z.
\end{equation}

Thus
\begin{gather*}
\|M_{\a} f\|_{L_{\Psi}(B)} \lesssim \|f\|_{L_{\Phi}(B(x,2r))} + \frac{1}{\Psi^{-1}\big(r^{-n}\big)} \,
\left(\sup_{t>2r}\frac{1}{|B(x,t)|^{1-\frac{\a}{n}}}\int_{B(x,t)}|f(z)|d z\right).
\end{gather*}

Let now $\Phi$ dominates globally the function $Q$. It is obvious that
\begin{gather*}
\|M_{\a} f\|_{WL_{\Psi}(B)} \lesssim \|M_{\a} f_1\|_{WL_{\Psi}(B)}+ \|M_{\a} f_2\|_{WL_{\Psi}(B)}
\end{gather*}
for every ball $B=B(x,r).$

By the boundedness of the operator $M_{\a}$ from $L_{\Phi}(\Rn)$ to $WL_{\Psi}(\Rn)$, provided by Theorem
\ref{bounFrMaxOrl},  we have
$$
\|M_{\a} f_1\|_{WL_{\Psi}(B)} \lesssim \|f\|_{L_{\Phi}(B(x,2r))}.
$$
Then by \eqref{ves1} we get the inequality \eqref{eq100WZ}.

\end{proof}

\begin{lem}\label{lem4.2.}
Let $\Phi$ and $\Psi$ Young functions, $f\in L_{\Phi}^{\rm loc}(\Rn)$ and $B=B(x,r)$.

If $(\Phi, \Psi)$ satisfy the conditions \eqref{condweakFrM}, then
\begin{align}\label{eq100WX}
\|M_{\a} f\|_{WL_{\Psi}(B)} \lesssim \frac{1}{\Psi^{-1}\big(r^{-n}\big)} \, \sup_{t>2r} \Psi^{-1}\big(t^{-n}\big) \, \|f\|_{L_{\Phi}(B(x,t))}.
\end{align}

If $(\Phi, \Psi)$ satisfy the conditions \eqref{condstrFrM}, then
\begin{equation}\label{4.5}
\|M_{\a} f\|_{L_{\Psi}(B)} \lesssim  \frac{1}{\Psi^{-1}\big(r^{-n}\big)} \, \sup_{t>2r} \Psi^{-1}\big(t^{-n}\big) \, \|f\|_{L_{\Phi}(B(x,t))}.
\end{equation}
\end{lem}

\begin{proof}
Suppose that the condition \eqref{condstrFrM} satisfied. Denote
\begin{equation*}
\begin{split}
\mathcal{M}_1:&=\frac{1}{\Psi^{-1}\big(r^{-n}\big)} \, \left(\sup_{t>2r}\frac{1}{|B(x,t)|^{1-\frac{\a}{n}}} \int_{B(x,t)}|f(z)|d z\right),
\\
\mathcal{M}_2:&=\|f\|_{L_{\Phi}(B(x,2r))}.
\end{split}
\end{equation*}
By Lemma \ref{lemHold}, we get
$$
\mathcal{M}_1 \lesssim \frac{1}{\Psi^{-1}\big(r^{-n}\big)} \,  \sup_{t>2r} t^{\a} \Phi^{-1}\big(t^{-n}\big) \|f\|_{L_{\Phi}(B(x,t))}.
$$
On the other hand, the conditions \eqref{condstrFrM} implies the condition \eqref{condweakFrM}. Since from Theorem \ref{bounFrMaxOrl}
$$\eqref{condstrFrM}\Rightarrow M_{\a} \text{  strong type  } (\Phi, \Psi) \Rightarrow M_{\a} \text{  weak type  } (\Phi, \Psi) \Rightarrow \eqref{condweakFrM}.$$
The condition \eqref{condweakFrM} is equivalent the condition $\Phi^{-1}(t)\lesssim t^{\frac{\a}{n}}\Psi^{-1}(t)$. Indeed,
\begin{eqnarray*}
{Q}^{-1}(t)&=&\inf\{r\geq0: {Q}(r)>t\}\\
{}&\geq&\inf\{r\geq0: \Phi(Cr)>t\}\\
{}&=&\frac{1}{C}\inf\{Cr\geq0: \Phi(Cr)>t\}\\
{}&=&\frac{1}{C}\Phi^{-1}(t).
\end{eqnarray*}
So we arrive
$$
\mathcal{M}_1 \lesssim \frac{1}{\Psi^{-1}\big(r^{-n}\big)} \,  \sup_{t>2r} \Psi^{-1}\big(t^{-n}\big) \|f\|_{L_{\Phi}(B(x,t))}.
$$
On the other hand
$$
\hskip-3cm\frac{1}{\Psi^{-1}\big(r^{-n}\big)} \,  \sup_{t>2r} \Psi^{-1}\big(t^{-n}\big) \|f\|_{L_{\Phi}(B(x,t))}
$$
\begin{equation}\label{ves2F}
\hskip+1cm\gtrsim \frac{1}{\Psi^{-1}\big(r^{-n}\big)} \,  \sup_{t>2r} \Psi^{-1}\big(t^{-n}\big) \, \|f\|_{L_{\Phi}(B(x,2r))}\thickapprox \mathcal{M}_2.
\end{equation}
Since $
\|M_{\a} f\|_{L_{\Psi}(B)}\leq \mathcal{M}_1+\mathcal{M}_2
$ by Lemma \ref{lem4.1},
we arrive at \eqref{4.5}.

Suppose that the condition \eqref{condweakFrM} satisfied. The inequality \eqref{eq100WX} directly follows from \eqref{eq100WZ}.
\end{proof}
If we take $\Phi(t)=t^{p},\,\Psi(t)=t^{q},\,1\le p,q<\i$ at Lemma \ref{lem4.2.} we get following estimates which was proved at \cite{GulShu}.
\begin{cor}
Let $f\in L_{p}^{\rm loc}(\Rn)$,  $0\le \a < n$,
$1\le p<\infty$, $1/p-1/q=\a/n$. Then
 $$
\|M_{\a} f\|_{L_{q}(B(x_0,r))}\lesssim r^{\frac{n}{q}} \sup_{t>2r}  t^{-\frac{n}{p}}\|f\|_{L_{p}(B(x_0,t))}, ~~~ 1<p\le q<\infty
$$
and
$$
\|M_{\a} f\|_{WL_{q}(B(x_0,r))}\lesssim r^{\frac{n}{q}} \sup_{t>2r} t^{-\frac{n}{q}}\|f\|_{L_{1}(B(x_0,t))}, ~~~ 1\le q<\infty.
$$\end{cor}

\begin{thm}\label{thm4.4.FrMax}
Let $0\le\a<n$ and the functions $(\varphi_1,\varphi_2)$ and $(\Phi, \Psi)$ satisfy the condition
\begin{equation}\label{eq3.6.VZfrMax}
\sup_{r<t<\infty} \Psi^{-1}\big(t^{-n}\big) \es_{t<s<\i}\frac{\varphi_1(x,s)}{\Phi^{-1}\big(s^{-n}\big)} \le C \, \varphi_2(x,r),
\end{equation}
where $C$ does not depend on $x$ and $r$.
Then for the conditions \eqref{condstrFrM},  the fractional maximal operator $M_{\a}$ is bounded from $M_{\Phi,\varphi_1}(\Rn)$ to $M_{\Psi,\varphi_2}(\Rn)$
and for the conditions \eqref{condweakFrM}, it is bounded from $M_{\Phi,\varphi_1}(\Rn)$ to $WM_{\Psi,\varphi_2}(\Rn)$.
\end{thm}
\begin{proof}
By Lemma \ref{lem4.2.} and Theorem \ref{thm5.1} we get
\begin{equation*}
\begin{split}
\|M_{\a} f\|_{M_{\Psi,\varphi_2}} & \lesssim \sup\limits_{x\in\Rn, r>0}
\varphi_2(x,r)^{-1}\,
\sup_{t>r} \Psi^{-1}\big(t^{-n}\big) \, \|f\|_{L_{\Phi}(B(x,t))}
\\
& \lesssim \sup\limits_{x\in\Rn, r>0}
\varphi_1(x,r)^{-1} \Phi^{-1}\big(r^{-n}\big) \|f\|_{L_{\Phi}(B(x,r))}
\\
& = \|f\|_{M_{\Phi,\varphi_1}},
\end{split}
\end{equation*}
if \eqref{condstrFrM} satisfied and
\begin{equation*}
\begin{split}
\|M_{\a} f\|_{WM_{\Psi,\varphi_2}} & \lesssim \sup\limits_{x\in\Rn, r>0}
\varphi_2(x,r)^{-1}\,
\sup_{t>r} \Psi^{-1}\big(t^{-n}\big) \, \|f\|_{L_{\Phi}(B(x,t))}
\\
& \lesssim \sup\limits_{x\in\Rn, r>0}
\varphi_1(x,r)^{-1} \Phi^{-1}\big(r^{-n}\big) \|f\|_{L_{\Phi}(B(x,r))}
\\
& = \|f\|_{M_{\Phi,\varphi_1}},
\end{split}
\end{equation*}
if \eqref{condweakFrM} satisfied.
\end{proof}

Note that analogue of the Theorem \ref{thm4.4.FrMax} for the Riesz potential proved in \cite{GulDerJFSA} as follows.
\begin{thm}\label{thm4.4.}
Let $0<\a<n$ and the functions $(\varphi_1,\varphi_2)$ and $(\Phi, \Psi)$ satisfy the condition
\begin{equation}\label{eq3.6.VZPot}
\int_{r}^{\i} \es_{t<s<\i}\frac{\varphi_1(x,s)}{\Phi^{-1}\big(s^{-n}\big)}\Psi^{-1}\big(t^{-n}\big)\frac{dt}{t}  \le C \, \varphi_2(x,r),
\end{equation}
where $C$ does not depend on $x$ and $r$.
Then for the conditions \eqref{condstrpot}, $I_{\a}$ is bounded from $M_{\Phi,\varphi_1}(\Rn)$ to $M_{\Psi,\varphi_2}(\Rn)$
and for the conditions \eqref{condweakpot}, $I_{\a}$ is bounded from $M_{\Phi,\varphi_1}(\Rn)$ to $WM_{\Psi,\varphi_2}(\Rn)$.
\end{thm}

\begin{rem}\label{supzyg} The condition \eqref{eq3.6.VZfrMax} is weaker than  \eqref{eq3.6.VZPot}. Indeed, \eqref{eq3.6.VZPot}
implies \eqref{eq3.6.VZfrMax}:
\begin{align*}
\varphi_2(x,r) & \gtrsim \int_{r}^{\i} \es_{t<\tau<\i}\frac{\varphi_1(x,\tau)}{\Phi^{-1}\big(\tau^{-n}\big)}\Psi^{-1}\big(t^{-n}\big)\frac{dt}{t}
\\
& \gtrsim \int_{s}^{\i} \es_{t<\tau<\i}\frac{\varphi_1(x,\tau)}{\Phi^{-1}\big(\tau^{-n}\big)}\Psi^{-1}\big(t^{-n}\big)\frac{dt}{t}
\\
& \gtrsim \es_{s<\tau<\i}\frac{\varphi_1(x,\tau)}{\Phi^{-1}\big(\tau^{-n}\big)} \int_s^{\infty}
\Psi^{-1}\big(t^{-n}\big)\frac{dt}{t}
\\
& \thickapprox  \es_{s<\tau<\i}\frac{\varphi_1(x,\tau)}{\Phi^{-1}\big(\tau^{-n}\big)} \Psi^{-1}\big(s^{-n}\big),
\end{align*}
where we took  $s \in (r,\infty)$, so that
\begin{equation*}
\sup\limits_{s>r} \es_{s<\tau<\i}\frac{\varphi_1(x,\tau)}{\Phi^{-1}\big(\tau^{-n}\big)} \Psi^{-1}\big(s^{-n}\big) \lesssim \varphi_2(x,r).
\end{equation*}
On the other hand the functions $\varphi_1(x,t)=\frac{\Phi^{-1}\big(t^{-n}\big)}{\Psi^{-1}\big(t^{-n}\big)}$ and $\varphi_2(x,t)=1$
satisfy the condition \eqref{eq3.6.VZfrMax}, but do not satisfy the condition \eqref{eq3.6.VZPot}.
\end{rem}

Consider the case $\a=0$ and $\Phi=\Psi$. In this case condition \eqref{condweakFrM} satisfied by any Young function and condition \eqref{condstrFrM} satisfied if and only if $\Phi\in\nabla_2$ (see \cite{Cianchi1, Kita2} for details). Therefore we get the following corollary for Hardy-Littlewood maximal operator which was proved in \cite{DGS}.
\begin{cor}
Let the functions $\varphi_1,\varphi_2$ and $ \Phi$ satisfy the condition
\begin{equation}\label{eq3.6.VZMax}
\sup_{r<t<\infty} \Phi^{-1}\big(t^{-n}\big) \es_{t<s<\i}\frac{\varphi_1(x,s)}{\Phi^{-1}\big(s^{-n}\big)} \le C \, \varphi_2(x,r),
\end{equation}
where $C$ does not depend on $x$ and $r$.
Then for $\Phi \in \nabla_2$,  the maximal operator $M$ is bounded from $M_{\Phi,\varphi_1}(\Rn)$ to $M_{\Phi,\varphi_2}(\Rn)$
and for every Young function $\Phi$, it is bounded from $M_{\Phi,\varphi_1}(\Rn)$ to $WM_{\Phi,\varphi_2}(\Rn)$.
\end{cor}

If we take $\Phi(t)=t^{p},\,\Psi(t)=t^{q},\,1\le p,q<\i$ at Theorem \ref{thm4.4.FrMax} we get the Spanne-Guliyev type result which was proved in \cite{GulShu}.
\begin{cor}
Let $0\le\a<n$, $1\le p < \frac{n}{\a}$, $\frac{1}{q}=\frac{1}{p}-\frac{\a}{n}$, and $(\varphi_1,\varphi_2)$ satisfy the condition
\begin{equation}\label{eq4.6.GSfrMax}
\sup_{r<t<\infty} \frac{\es_{t<s<\i}\varphi_{1}(x,s)s^{\frac{n}{p}}}{t^{\frac{n}{q}}}\le C\varphi_{2}(x,r),
\end{equation}
where $C$ does not depend on $x$ and $r$. Then for $p>1$, $M_{\a}$ is bounded from $M_{p,\varphi_1}(\Rn)$ to $M_{q,\varphi_2}(\Rn)$
and for  $p=1$, it is bounded from $M_{1,\varphi_1}(\Rn)$ to $WM_{q,\varphi_2}(\Rn)$.
\end{cor}

In the case $\varphi_1(x,r)=\frac{\Phi^{-1}\big(r^{-n}\big)}{\Phi^{-1}\big(r^{-\lambda_1}\big)}$, $\varphi_2(x,r)=\frac{\Psi^{-1}\big(r^{-n}\big)}{\Psi^{-1}\big(r^{-\lambda_2}\big)}$  from Theorem \ref{thm4.4.FrMax} we get the following Spanne type theorem for the boundedness of the fractional maximal operator on Orlicz-Morrey spaces.
\begin{cor} \label{ggffdd}
Let $0\le\a<n$, $\Phi$ and $\Psi$ be Young functions, $0 \le \lambda_1, \lambda_2 < n$ and $(\Phi, \Psi)$ satisfy the condition
\begin{equation}\label{eq3.6.VZfrMaxcor}
\sup_{r<t<\infty} \frac{\Psi^{-1}\big(t^{-n}\big)}{\Phi^{-1}\big(t^{-\lambda_1}\big)}  \le C \, \frac{\Psi^{-1}\big(r^{-n}\big)}{\Psi^{-1}\big(r^{-\lambda_2}\big)},
\end{equation}
where $C$ does not depend on $r$.
Then for the conditions \eqref{condstrFrM}, $M_{\a}$ is bounded from $M_{\Phi,\lambda_1}(\Rn)$ to $M_{\Psi,\lambda_2}(\Rn)$
and for the conditions \eqref{condweakFrM}, $M_{\a}$ is bounded from $M_{\Phi,\lambda_1}(\Rn)$ to $WM_{\Psi,\lambda_2}(\Rn)$.
\end{cor}

\begin{rem}
If we take $\Phi(t)=t^{p},\,\Psi(t)=t^{q},\,1\le p,q<\i$ at Corollary \ref{ggffdd} we get Spanne type boundedness of $M_{\a}$, i.e. if $0\le\a<n$, $1<p<\frac{n}{\a}$, $0<\lambda<n-\a p$, $\frac{1}{p}-\frac{1}{q}=\frac{\a}{n}$ and $\frac{\lambda}{p}=\frac{\mu}{q}$, then for $p>1$ $M_{\a}$ is bounded from $M_{p,\lambda}(\Rn)$ to $M_{q,\mu}(\Rn)$ and for $p=1$, $M_{\a}$ is bounded from $M_{1,\lambda}(\Rn)$ to $WM_{q,\mu}(\Rn)$.
\end{rem}



\section{Commutators of the fractional maximal operator in the spaces $M_{\Phi,\varphi}$}

The theory of commutator was originally studied by, Coifman, Rochberg and Weiss in \cite{CRW}. Since then, many authors have been interested in studying this theory.

We recall the definition of the space of $BMO(\Rn)$.

\begin{defn}
Suppose that $f\in L_1^{\rm loc}(\Rn)$, let
\begin{equation*}
\|f\|_\ast=\sup_{x\in\Rn, r>0}\frac{1}{|B(x,r)|}
\int_{B(x,r)}|f(y)-f_{B(x,r)}|dy<\infty,
\end{equation*}
where
$$
f_{B(x,r)}=\frac{1}{|B(x,r)|} \int_{B(x,r)} f(y)dy.
$$
Define
$$
BMO(\Rn)=\{ f\in L_1^{\rm loc}(\Rn) ~ : ~ \| f \|_{\ast} < \infty  \}.
$$
\end{defn}

Modulo constants, the space $BMO(\Rn)$ is a Banach space with respect to the norm $\| \cdot \|_{\ast}$.

\begin{rem} \label{rem2.4.}
$(1)~~$ The John--Nirenberg inequality:
there are constants $C_1$, $C_2>0$, such that for all $f \in BMO(\Rn)$ and $\beta>0$
$$
\left| \left\{ x \in B \, : \, |f(x)-f_{B}|>\beta \right\}\right|
\le C_1 |B| e^{-C_2 \beta/\| f \|_{\ast}}, ~~~ \forall B \subset \Rn.
$$

$(2)~~$ The John--Nirenberg inequality implies that
\begin{equation} \label{lem2.4.}
\|f\|_\ast \thickapprox \sup_{x\in\Rn, r>0}\left(\frac{1}{|B(x,r)|}
\int_{B(x,r)}|f(y)-f_{B(x,r)}|^p dy\right)^{\frac{1}{p}}
\end{equation}
for $1<p<\infty$.

$(3)~~$ Let $f\in BMO(\Rn)$. Then there is a constant $C>0$ such that
\begin{equation} \label{propBMO}
\left|f_{B(x,r)}-f_{B(x,t)}\right| \le C \|f\|_\ast \ln \frac{t}{r} \;\;\; \mbox{for} \;\;\; 0<2r<t,
\end{equation}
where $C$ is independent of $f$, $x$, $r$ and $t$.
\end{rem}

\begin{defn}
A Young function $\Phi$ is said to be of upper type p (resp. lower type p) for some $p\in[0,\i)$, if there exists a positive constant $C$ such that, for all $t\in[1,\i)$(resp. $t\in[0,1]$) and $s\in[0,\i)$,
$$
\Phi(st)\le Ct^p\Phi(s).
$$
\end{defn}

\begin{rem}\label{remlowup}
We know that if $\Phi$ is lower type $p_0$ and upper type $p_1$ with $1<p_0\le p_1<\i$, then $\Phi\in \Delta_2\cap\nabla_2$. Conversely if $\Phi\in \Delta_2\cap\nabla_2$, then  $\Phi$ is lower type $p_0$ and upper type $p_1$ with $1<p_0\le p_1<\i$ (see \cite{KokKrbec}).
\end{rem}

\begin{lem}\cite{Ky1}\label{Kylowupp}
Let $\Phi$ be a Young function which is lower type $p_0$ and upper type $p_1$ with $0<p_0\le p_1<\i$. Let $\widetilde{C}$ be a positive constant. Then there exists a positive constant $C$ such that for any ball $B$ of $\Rn$ and $\mu\in(0,\i)$
$$\int_{B}\Phi\left(\frac{|f(x)|}{\mu}\right)dx\le \widetilde{C}$$
implies that $\|f\|_{L_\Phi(B)}\le C\mu$.
\end{lem}

\begin{lem}\label{Bmo-orlicz}
Let $f\in BMO(\Rn)$ and $\Phi$ be a Young function. Let $\Phi$ is lower type $p_0$ and upper type $p_1$ with $1<p_0\le p_1<\i$, then
$$
\|f\|_\ast \thickapprox \sup_{x\in\Rn, r>0}\Phi^{-1}\big(r^{-n}\big)\left\|f(\cdot)-f_{B(x,r)}\right\|_{L_{\Phi}(B(x,r))}
$$
\end{lem}

\begin{proof}
By H\"{o}lder's inequality, we have
$$\|f\|_\ast \lesssim \sup_{x\in\Rn, r>0}\Phi^{-1}\big(r^{-n}\big)\left\|f(\cdot)-f_{B(x,r)}\right\|_{L_{\Phi}(B(x,r))}.$$

Now we show that
$$
\sup_{x\in\Rn, r>0}\Phi^{-1}\big(r^{-n}\big)\left\|f(\cdot)-f_{B(x,r)}\right\|_{L_{\Phi}(B(x,r))} \lesssim \|f\|_\ast .
$$
Without loss of generality, we may assume that $\|f\|_\ast=1$; otherwise,
we replace $f$ by $f/\|f\|_\ast$. By the fact that $\Phi$ is lower type $p_0$ and upper type $p_1$ and \eqref{younginverse} it follows that
$$
\int_{B(x,r)}\Phi\left(\frac{|f(y)-f_{B(x,r)}|\Phi^{-1}\big(|B(x,r)|^{-1}\big)}{\|f\|_\ast}\right)dy
$$
$$
=\int_{B(x,r)}\Phi\left(|f(y)-f_{B(x,r)}|\Phi^{-1}\big(|B(x,r)|^{-1}\big)\right)dy
$$
$$
\lesssim\frac{1}{|B(x,r)|}\int_{B(x,r)}\left[|f(y)-f_{B(x,r)}|^{p_0}+|f(y)-f_{B(x,r)}|^{p_1}\right]dy\lesssim 1.
$$
By Lemma \ref{Kylowupp} we get the desired result.
\end{proof}

\begin{defn}
Let $\Phi$ be a Young function. Let
$$
a_{\Phi}:=\inf_{t\in(0,\i)}\frac{t\Phi^{\prime}(t)}{\Phi(t)}, \qquad  b_{\Phi}:=\sup_{t\in(0,\i)}\frac{t\Phi^{\prime}(t)}{\Phi(t)}.
$$
\end{defn}

\begin{rem}\label{indorl}
It is known that $\Phi\in \Delta_2\cap\nabla_2$ if and only if $1<a_{\Phi}\le b_{\Phi}<\i$ (See \cite{KrasnRut}).
\end{rem}

\begin{rem}
Remark \ref{indorl} and Remark \ref{remlowup} show us that a Young function $\Phi$ is lower type $p_0$ and upper type $p_1$ with $1<p_0\le p_1<\i$ if and only if $1<a_{\Phi}\le b_{\Phi}<\i$.
\end{rem}

The commutators generated by  $b\in  L^1_{\rm loc}(\Rn)$ and the operators $M_{\a}$ and  $I_{\a}$ are  defined by
\begin{equation}
M_{b,\a}(f)(x)=\sup_{t>0}|B(x,t)|^{-1+\frac{\a}{n}} \, \int _{B(x,t)}|b(x)-b(y)||f(y)|dy,
\end{equation}
\begin{equation}
[b,I_{\a}]f(x)=\int _{\Rn}\frac{b(x)-b(y)}{|x-y|^{n-\alpha}}f(y)dy,
\end{equation}
\begin{equation}
|b,I_{\alpha}|f(x):=\int_{\Rn}\frac{|b(x)-b(y)|}{|x-y|^{n-\alpha}}f(y)dy,
\end{equation}
respectively.




The known boundedness statements for the commutator operators $[b, I_{\alpha }]$ and $|b,I_{\alpha}|$ in Orlicz spaces run as follows.

\begin{thm}\label{TinOrlicz} \cite{FuYangYuan}
Let $0<\alpha <n$ and $b\in BMO(\Rn)$. Let $\Phi$ be a Young function and $\Psi$
defined by  its inverse  $\Psi^{-1}(t):=\Phi^{-1}(t)t^{-\alpha /n}$. If $1<a_{\Phi}\le b_{\Phi}<\infty$
and $1<a_{\Psi}\le b_{\Psi}<\infty$, then $[b,I_{\alpha }]$ is bounded from $L^{\Phi}(\Rn)$ to $L^{\Psi}(\Rn)$.
\end{thm}

\begin{rem}\label{lurem3}
Note that, the operator $|b,I_{\alpha}|$ is bounded from $L^{\Phi}(\Rn)$ to $L^{\Psi}(\Rn)$ under the conditions of Theorem \ref{TinOrlicz}.
Indeed, the proof of this fact is similar to proof of Theorem \ref{TinOrlicz}.
\end{rem}

In \cite{FuYangYuan1} it was  proved that the commutator of the Hardy-Littlewood maximal operator $M_{b}$
with $b\in BMO(\Rn)$, is bounded in $L^\Phi(\Rn)$ for any Young function $\Phi$ with $1<a_{\Phi}\le b_{\Phi}<\infty$. This result together with the well known inequality
$ M_{\alpha, b}(f)(x)\lesssim |b,I_{\alpha}| (|f|)(x)$ and Remark \ref{lurem3}, imply the following theorem.

\begin{thm}\label{ggsgsd}
Let $0\le\alpha <n$ and $b, \Phi$ and $\Psi$ the same as in Theorem \ref{TinOrlicz}.  If $1<a_{\Phi}\le b_{\Phi}<\infty $ and $1<a_{\Psi}\le b_{\Psi}<\infty ,$ then $M_{b,\alpha}$ is bounded from $L^{\Phi}(\Rn)$ to $L^{\Psi}(\Rn)$.
\end{thm}

The following lemma is valid.
\begin{lem}\label{lem5.1.}
Let $0\le\a<n$ and $b\in BMO(\Rn)$. Let $\Phi$ be a Young function and $\Psi$ defined, via its inverse, by setting, for all $t\in(0,\i)$, $\Psi^{-1}(t):=\Phi^{-1}(t)t^{-\a/n}$ and $1<a_{\Phi}\le b_{\Phi}<\i$ and $1<a_{\Psi}\le b_{\Psi}<\i$,
then the inequality
\begin{equation*}\label{eq5.1.}
\|M_{b,\a}f\|_{L_\Psi(B(x_0,r))} \lesssim \|b\|_{*} \, \frac{1}{\Psi^{-1}\big(r^{-n}\big)}
 \sup_{t>2r} \Big(1+\ln \frac{t}{r}\Big) \Psi^{-1}\big(t^{-n}\big) \|f\|_{L_{\Phi}(B(x_0,t))}
\end{equation*}
holds for any ball $B(x_0,r)$ and for all $f\in L_{\Phi}^{\rm loc}(\Rn)$.

\end{lem}

\begin{proof}
For arbitrary $x_0 \in\Rn$, set $B=B(x_0,r)$ for the ball
centered at $x_0$ and of radius $r$. Write $f=f_1+f_2$ with
$f_1=f\chi_{_{2B}}$ and $f_2=f\chi_{_{\dual (2B)}}$. Hence
$$
\left\|M_{b,\a}f\right\|_{L_\Psi(B)} \leq
\left\|M_{b,\a}f_1 \right\|_{L_\Psi(B)}+
\left\|M_{b,\a}f_2 \right\|_{L_\Psi(B)}.
$$
From the boundedness of $M_{b,\a}$ from $L_{\Phi}(\Rn)$ to $L_{\Psi}(\Rn)$  
it follows that
\begin{align*}
\|M_{b,\a}f_1\|_{L_\Psi(B)} & \leq
\|M_{b,\a}f_1\|_{L_\Psi(\Rn)}
\\
& \lesssim  \|b\|_{*} \, \|f_1\|_{L_\Phi(\Rn)} = \|b\|_{*} \, \|f\|_{L_\Phi(2B)}.
\end{align*}

For $x \in B$ we have
\begin{align*}
M_{b,\a}f_2(x) & =\sup_{t>0} \frac{1}{|B(x,t)|^{1-\frac{\a}{n}}} \int_{B(x,t)}|b(y)-b(x)||f_2(y)|dy
\\
&=\sup_{t>0} \frac{1}{|B(x,t)|^{1-\frac{\a}{n}}}\int_{B(x,t)\cap\dual (2B)}|b(y)-b(x)||f(y)|dy.
\end{align*}

Let $x$ be an arbitrary point from $B$. If $B(x,t)\cap\{\dual (2B)\}\neq\emptyset,$ then $t>r$. Indeed, if $y\in B(x,t)\cap  \{\dual (2B)\},$
then $t > |x-y| \geq |x_0-y|-|x_0-x|>2r-r=r$.

On the other hand, $B(x,t)\cap  \{\dual (2B)\}\subset B(x_0,2t)$. Indeed, if  $y\in B(x,t)\cap  \{\dual (2B)\},$ then
we get $|x_0-y|\leq |x-y|+|x_0-x|<t+r<2t$.

Hence
\begin{align*}
M_{b,\a}(f_2)(x) & = \sup_{t>0} \frac{1}{|B(x,t)|^{1-\frac{\a}{n}}}\int_{B(x,t)\cap\dual (2B)}|b(y)-b(x)||f(y)|dy
\\
&\leq 2^{n-\a}\sup_{t>r} \frac{1}{|B(x_0,2t)|^{1-\frac{\a}{n}}}\int_{B(x_0,2t)}|b(y)-b(x)||f(y)|dy
\\
&=2^{n-\a}\sup_{t>2r} \frac{1}{|B(x_0,t)|^{1-\frac{\a}{n}}}\int_{B(x_0,t)}|b(y)-b(x)||f(y)|dy
\end{align*}

Therefore, for all $x \in B$ we have

\begin{equation}\label{gulsuk6.3}
M_{b,\a}(f_2)(x) \leq 2^{n-\a}\sup_{t>2r} \frac{1}{|B(x_0,t)|^{1-\frac{\a}{n}}}\int_{B(x_0,t)}|b(y)-b(x)||f(y)|dy.
\end{equation}

Then
\begin{align*}
\|M_{b,\a}f_2\|_{L_\Psi(B)} & \lesssim
\left\|\sup_{t>2r} \frac{1}{|B(x_0,t)|^{1-\frac{\a}{n}}}\int_{B(x_0,t)}|b(y)-b(\cdot)||f(y)|dy\right\|_{L_\Psi(B)}
\\
&\lesssim \left\|\sup_{t>2r} \frac{1}{|B(x_0,t)|^{1-\frac{\a}{n}}}\int_{B(x_0,t)}|b(y)-b_B||f(y)|dy\right\|_{L_\Psi(B)}
\\
&\quad +\left\|\sup_{t>2r} \frac{1}{|B(x_0,t)|^{1-\frac{\a}{n}}}\int_{B(x_0,t)}|b(\cdot)-b_B||f(y)|dy\right\|_{L_\Psi(B)}
\\
&=J_1+J_2.
\end{align*}
Let us estimate $J_1$.
\begin{align*}
J_1 & = \frac{1}{\Psi^{-1}\big(r^{-n}\big)}\sup_{t>2r} \frac{1}{|B(x_0,t)|^{1-\frac{\a}{n}}}\int_{B(x_0,t)}|b(y)-b_B||f(y)|dy
\\
&\thickapprox \frac{1}{\Psi^{-1}\big(r^{-n}\big)}\sup_{t>2r} t^{\a-n}\int_{B(x_0,t)}|b(y)-b_B||f(y)|dy.
\end{align*}

Applying H\"older's inequality, by Lemma \ref{Bmo-orlicz} and \eqref{propBMO}  we get
\allowdisplaybreaks
\begin{align*}
J_1 & \lesssim \frac{1}{\Psi^{-1}\big(r^{-n}\big)}\sup_{t>2r} t^{\a-n}\int_{B(x_0,t)}|b(y)-b_{B(x_0,t)}||f(y)|dy
\\
&\quad + \frac{1}{\Psi^{-1}\big(r^{-n}\big)}\sup_{t>2r} t^{\a-n}|b_{B(x_0,r)}-b_{B(x_0,t)}|\int_{B(x_0,t)}|f(y)|dy
\\
&
\lesssim\frac{1}{\Psi^{-1}\big(r^{-n}\big)}\sup_{t>2r} t^{\a-n}\left\|b(\cdot)-b_{B(x_0,t)}\right\|_{L_{\widetilde{\Phi}}(B(x_0,t))} \|f\|_{L_\Phi(B(x_0,t))}
\\
& \quad + \frac{1}{\Psi^{-1}\big(r^{-n}\big)} \sup_{t>2r} t^{\a-n}|b_{B(x_0,r)}-b_{B(x_0,t)}|t^n\Phi^{-1}\big(t^{-n}\big)\|f\|_{L_\Phi(B(x_0,t))}
\\
& \lesssim \|b\|_{*}\,\frac{1}{\Psi^{-1}\big(r^{-n}\big)}
\sup_{t>2r}\Psi^{-1}\big(t^{-n}\big)\Big(1+\ln \frac{t}{r}\Big)
\|f\|_{L_\Phi(B(x_0,t))}.
\end{align*}
In order to estimate $J_2$ note that
\begin{align*}
J_2 & \thickapprox
\left\|b(\cdot)-b_{B}\right\|_{L_\Psi(B)}\sup_{t>2r} t^{\a-n}\int_{B(x_0,t)}|f(y)|dy
\\
&\lesssim \|b\|_{*}\,\frac{1}{\Psi^{-1}\big(r^{-n}\big)}\sup_{t>2r}\Psi^{-1}\big(t^{-n}\big)\|f\|_{L_\Phi(B(x_0,t))}
\end{align*}
Summing up $J_1$ and $J_2$ we get
\begin{equation} \label{deckfV}
\|M_{b,\a}f_2\|_{L_\Psi(B)}
\lesssim \|b\|_{*}\,\frac{1}{\Psi^{-1}\big(r^{-n}\big)}
\sup_{t>2r}\Psi^{-1}\big(t^{-n}\big)\Big(1+\ln \frac{t}{r}\Big)
\|f\|_{L_\Phi(B(x_0,t))}.
\end{equation}
Finally,
$$
\|M_{b,\a}f\|_{L_\Psi(B)}\lesssim \|b\|_{*}\,\|f\|_{L_\Phi(2B)}+
\|b\|_{*}\,\frac{1}{\Psi^{-1}\big(r^{-n}\big)}
\sup_{t>2r}\Psi^{-1}\big(t^{-n}\big)\Big(1+\ln \frac{t}{r}\Big)
\|f\|_{L_\Phi(B(x_0,t))},
$$
and the statement of Lemma \ref{lem5.1.} follows by \eqref{ves2F}.
%
\end{proof}
\begin{thm}\label{3.4.Xcom}
Let $0\le\a<n$ and $b\in BMO(\Rn)$. Let also $\Phi$ be a Young function and $\Psi$ defined, via its inverse, by setting, for all $t\in(0,\i)$, $\Psi^{-1}(t):=\Phi^{-1}(t)t^{-\a/n}$ and $1<a_{\Phi}\le b_{\Phi}<\i$ and $1<a_{\Psi}\le b_{\Psi}<\i$, $(\varphi_1,\varphi_2)$ and $(\Phi, \Psi)$ satisfy the condition
\begin{equation}\label{eq3.6.VZfrMaxcom}
\sup_{r<t<\infty} \Big(1+\ln \frac{t}{r}\Big) \Psi^{-1}\big(t^{-n}\big) \es_{t<s<\i}\frac{\varphi_1(x,s)}{\Phi^{-1}\big(s^{-n}\big)} \le C \, \varphi_2(x,r),
\end{equation}
where $C$ does not depend on $x$ and $r$.
Then  the operator $M_{b,\a}$ is bounded from $M_{\Phi,\varphi_1}(\Rn)$ to
$M_{\Psi,\varphi_2}(\Rn)$. Moreover
$$\|M_{b,\a} f\|_{M_{\Psi,\varphi_2}} \le  \|b\|_{*}\|f\|_{M_{\Phi,\varphi_1}}.$$
%
\end{thm}
\begin{proof}
The statement of Theorem \ref{3.4.Xcom} follows by Lemma \ref{lem5.1.} and Theorem
\ref{thm5.1} in the same manner as in the proof of Theorem
\ref{thm4.4.FrMax}.
\end{proof}

If we take $\a=0$ at Theorem \ref{3.4.Xcom} we get the following new result for the commutator of Hardy-Littlewood maximal operator $M_{b}$.
\begin{cor}
Let $b\in BMO(\Rn)$, $\Phi$ be a Young function with $1<a_{\Phi}\le b_{\Phi}<\i$, $(\varphi_1,\varphi_2)$ and $\Phi$ satisfy the condition
\begin{equation}\label{eq3.6.VZMaxcom}
\sup_{r<t<\infty} \Big(1+\ln \frac{t}{r}\Big) \Phi^{-1}\big(t^{-n}\big) \es_{t<s<\i}\frac{\varphi_1(x,s)}{\Phi^{-1}\big(s^{-n}\big)} \le C \, \varphi_2(x,r),
\end{equation}
where $C$ does not depend on $x$ and $r$.
Then  the operator $M_{b}$ is bounded from $M_{\Phi,\varphi_1}(\Rn)$ to
$M_{\Phi,\varphi_2}(\Rn)$. Moreover
$$\|M_{b} f\|_{M_{\Phi,\varphi_2}} \le  \|b\|_{*}\|f\|_{M_{\Phi,\varphi_1}}.$$
\end{cor}

Note that analogue of the Theorem \ref{3.4.Xcom} for the commutator of the Riesz potential $[b,I_{\a}]$ proved in \cite{GulDerJFSA} as follows.
\begin{thm}\label{3.4.XcomT}
Let $0<\a<n$ and $b\in BMO(\Rn)$. Let $\Phi$ be a Young function and $\Psi$ defined, via its inverse, by setting, for all $t\in(0,\i)$, $\Psi^{-1}(t):=\Phi^{-1}(t)t^{-\a/n}$ and $1<a_{\Phi}\le b_{\Phi}<\i$ and $1<a_{\Psi}\le b_{\Psi}<\i$.  $(\varphi_1,\varphi_2)$ and $(\Phi, \Psi)$ satisfy the condition
\begin{equation}\label{eq3.6.VZpotcom}
\int_{r}^{\i}\Big(1+\ln \frac{t}{r}\Big) \es_{t<s<\i}\frac{\varphi_1(x,s)}{\Phi^{-1}\big(s^{-n}\big)}\Psi^{-1}\big(t^{-n}\big)\frac{dt}{t}  \le C \, \varphi_2(x,r),
\end{equation}
where $C$ does not depend on $x$ and $r$.

Then  the operator $[b,I_{\a}]$ is bounded from $M_{\Phi,\varphi_1}(\Rn)$ to
$M_{\Psi,\varphi_2}(\Rn)$. Moreover
$$\|[b,I_{\a}] f\|_{M_{\Psi,\varphi_2}} \le  \|b\|_{*}\|f\|_{M_{\Phi,\varphi_1}}.$$
%
\end{thm}

\begin{rem}
The condition \eqref{eq3.6.VZfrMaxcom} is weaker than \eqref{eq3.6.VZpotcom}. See Remark \ref{supzyg} for details.
\end{rem}

If we take $\Phi(t)=t^{p},\,\Psi(t)=t^{q},\,1< p,q<\i$ at Theorem \ref{3.4.Xcom} we get the Spanne-Guliyev type result which was proved at \cite{GulShu}.
\begin{cor}
Let $0\le\a<n$, $1< p < \frac{n}{\a}$, $\frac{1}{q}=\frac{1}{p}-\frac{\a}{n}$, and $(\varphi_1,\varphi_2)$ satisfy the condition
\begin{equation}\label{eq4.6.GSfrMax}
\sup_{r<t<\infty} \Big(1+\ln \frac{t}{r}\Big) \frac{\es_{t<s<\i}\varphi_{1}(x,s)s^{\frac{n}{p}}}{t^{\frac{n}{q}}}\le C\varphi_{2}(x,r),
\end{equation}
where $C$ does not depend on $x$ and $r$. Then $M_{b,\a}$ is bounded from $M_{p,\varphi_1}(\Rn)$ to $M_{q,\varphi_2}(\Rn)$.
\end{cor}

In the case $\varphi_1(x,r)=\frac{\Phi^{-1}\big(r^{-n}\big)}{\Phi^{-1}\big(r^{-\lambda_1}\big)}$, $\varphi_2(x,r)=\frac{\Psi^{-1}\big(r^{-n}\big)}{\Psi^{-1}\big(r^{-\lambda_2}\big)}$  from Theorem \ref{3.4.Xcom} we get the following Spanne type theorem for the boundedness of the operator $M_{b,\a}$ on Orlicz-Morrey spaces.
\begin{cor} \label{ggffddc}
Let $0\le\a<n$, $0 \le \lambda_1, \lambda_2 < n$ and $b\in BMO(\Rn)$. Let also $\Phi$ be a Young function and $\Psi$ defined, via its inverse, by setting, for all $t\in(0,\i)$, $\Psi^{-1}(t):=\Phi^{-1}(t)t^{-\a/n}$, $1<a_{\Phi}\le b_{\Phi}<\i$, $1<a_{\Psi}\le b_{\Psi}<\i$ and $(\Phi, \Psi)$ satisfy the condition
\begin{equation}\label{eq3.6.VZfrMaxcor}
\sup_{r<t<\infty} \Big(1+\ln \frac{t}{r}\Big)\frac{\Psi^{-1}\big(t^{-n}\big)}{\Phi^{-1}\big(t^{-\lambda_1}\big)}  \le C \, \frac{\Psi^{-1}\big(r^{-n}\big)}{\Psi^{-1}\big(r^{-\lambda_2}\big)},
\end{equation}
where $C$ does not depend on $r$.
Then $M_{b,\a}$ is bounded from $M_{\Phi,\lambda_1}(\Rn)$ to $M_{\Psi,\lambda_2}(\Rn)$.
\end{cor}

\begin{rem}
If we take $\Phi(t)=t^{p},\,\Psi(t)=t^{q},\,1\le p,q<\i$ at Corollary \ref{ggffddc} we get Spanne type boundedness of $M_{b,\a}$, i.e. if $0\le\a<n$, $1<p<\frac{n}{\a}$, $0<\lambda<n-\a p$, $\frac{1}{p}-\frac{1}{q}=\frac{\a}{n}$ and $\frac{\lambda}{p}=\frac{\mu}{q}$, then for $p>1$ $M_{b,\a}$ is bounded from $M_{p,\lambda}(\Rn)$ to $M_{q,\mu}(\Rn)$ and for $p=1$, $M_{b,\a}$ is bounded from $M_{1,\lambda}(\Rn)$ to $WM_{q,\mu}(\Rn)$.
\end{rem}

\

\

\end{document}